\documentclass[12pt]{amsart}

\usepackage{amsmath, amsthm, amssymb}
\input xy
\xyoption{all}
\usepackage{mathrsfs}
\usepackage{enumerate}
\usepackage{caption}
\usepackage{graphicx}
\usepackage{enumitem}
\usepackage{empheq}
\usepackage{tikz}

\newtheorem{theorem}[subsection]{Theorem}

\newtheorem{lemma}[subsection]{Lemma}

\theoremstyle{definition}

\newtheorem{example}[subsection]{Example}

\newcommand{\M}{\ensuremath{\overline{\mathcal{M}}}}

\renewcommand{\d}{\ensuremath{\partial}}

\renewcommand{\S}{\ensuremath{\widetilde{\mathcal{S}}}}

\usepackage{geometry}
\begin{document}

\title{Boundary complexes of moduli spaces of curves in higher genus}

\begin{abstract}
Given a collection of boundary divisors in the moduli space $\M_{0,n}$ of stable genus-zero $n$-pointed curves, Giansiracusa proved that their intersection is nonempty if and only if all pairwise intersections are nonempty. We give a complete classification of the pairs $(g,n)$ for which the analogous statement holds in $\M_{g,n}$. 
\end{abstract}

\author[E.~Clader]{Emily Clader}

\author[D.~Luber]{Dante Luber}

\author[K.~Quillin]{Kyla Quillin}

\maketitle

\section{Introduction}

The Deligne--Mumford moduli space $\M_{g,n}$ of stable curves is an object of intense study in a surprisingly broad array of contexts, from the enumerative geometry of curve-counting \cite{Hori} to the combinatorics of hyperplane arrangements \cite{DeConciniProcesi} to the birational geometry of the moduli space as a variety in its own right \cite{CastravetTevelev}.  Toward many of these ends, one would hope to understand the intersection theory of $\M_{g,n}$, and in particular to compute its Chow ring.

In genus zero, such a computation is possible: Keel \cite{Keel} gave an explicit presentation of $A^*(\M_{0,n})$ as generated by the {\it boundary divisors}, which are the closures of the loci of curves with a single node across which the distribution of the marked points is specified.  The analogous statement fails in higher genus, and indeed, a full understanding of the Chow ring of $\M_{g,n}$ seems out of reach; the top-codimension component of $A^*(\M_{1,11})$, for example, is known to be uncountably generated.  As a more manageable substitute, one can instead study the tautological ring, a subring of $A^*(\M_{g,n})$ that has a concrete and well-understood generating set and yet has been proven to contain nearly every Chow class of geometric interest.  An analysis of the boundary divisors, while no longer capturing the entire intersection theory of the moduli space, still plays a key role in understanding the tautological ring of $\M_{g,n}$.

One striking property of the boundary divisors in genus zero is that a a collection of boundary divisors $D_1, \ldots, D_k$ in $\M_{0,n}$ has nonempty intersection if and only if each pairwise intersection $D_i \cap D_j$ is nonempty.  This ``folklore" result was given a succinct proof by Giansiracusa \cite{Giansiracusa} by relating it to a theorem in phylogenetics.  The key observation behind this connection is that the dual graphs encoding boundary strata of $\M_{0,n}$ can be viewed, from another angle, as the phylogenetic trees encoding the evolutionary history of organisms.

The goal of the current paper is to study when the corresponding property holds in higher genus, where the dual graphs are no longer necessarily trees and thus the connection to phylogenetics is no longer available.  Our aim, in other words, is to classify the pairs $(g,n)$ for which, given any collection of boundary divisors in $\M_{g,n}$, nonempty pairwise intersection implies nonempty total intersection.  This is equivalent to the condition that the boundary complex of $\M_{g,n}$---a simplicial complex defined below---is a flag complex, meaning that it is the maximal simplicial complex on its $1$-skeleton.  Our main theorem is the following:

\begin{theorem}
\label{thm:main}
The boundary complex of $\M_{g,n}$ is a flag complex if and only if either $g \in \{0,1\}$, $n \in \{0,1\}$, or $g=n=2$.
\end{theorem}


\subsection*{Acknowledgments}  The authors would like to thank D. Ross for calling their attention to this problem, and F. Ardila and S. Hosten for feedback on earlier versions.  The first and second authors were partially supported by NSF DMS grant 1810969.

\section{Background on boundary strata}

\subsection{Preliminary definitions}

A point of the moduli space $\M_{g,n}$ parameterizes a tuple $(C; x_1, \ldots, x_n)$, where $C$ is an algebraic curve of arithmetic genus $g$ with at worst nodal singularities, the $x_i$ are distinct smooth points of $C$ (the {\it marked points}), and the tuple is {\it stable} in the sense that it has finitely many automorphisms.  Explicitly, stability is equivalent to the requirement that each irreducible component of $C$ of geometric genus zero have at least three special points (marked points or half-nodes) and each irreducible component of geometric genus one have at least one special point.

Any element $(C; x_1, \ldots, x_n)$ of $\M_{g,n}$ has an associated {\it dual graph}, which consists of
\begin{itemize}
\item a vertex $v_i$ for each irreducible component $C_i$ of $C$, decorated with the geometric genus $g(v_i)$ of $C_i$;
\item an edge between vertices $v_i$ and $v_j$ for each node joining $C_i$ to $C_j$;
\item a half-edge (or ``leg") at vertex $v_i$ for each of the marked points $x_k$ on $C_i$, decorated with the number $k \in \{1,\ldots, n\}$ of the marked point.
\end{itemize}
Conversely, given a dual graph $G$, there is an associated {\it boundary stratum} $S_G \subseteq \M_{g,n}$, which is the closure of the set of curves with dual graph $G$.  The codimension of $S_G$ is equal to the number of edges of $G$, so a {\it boundary divisor} of $\M_{g,n}$ is specified by a dual graph with a single edge. The boundary divisors have normal crossings \cite{Knudsen}, which implies in particular that an intersection of $k$ distinct boundary divisors has codimension $k$ whenever it is nonempty.

Inclusions among boundary strata are described in terms of degenerations of their associated dual graphs.  Specifically, let $G$ be a dual graph and let $e$ be an edge of $G$ between distinct vertices $v$ and $w$.  Then the {\it smoothing} of $G$ along $e$ is the graph obtained from $G$ by removing $e$ and replacing $v$ and $w$ by a single vertex of genus $g(v) + g(w)$.  Similarly, the smoothing of a dual graph along a self-edge $e$ at vertex $v$ is given by removing $e$ and increasing $g(v)$ to $g(v) +1$.  We say that $G$ is a {\it degeneration} of $H$ if $H$ can be obtained from $G$ by smoothing some subset of the edges, and we have
\begin{equation}
\label{eq:SGSH}
S_G \subseteq S_H \; \; \Leftrightarrow \;\; G \text{ is a degeneration of } H.
\end{equation}
In particular, for any dual graph $G$ and any edge $e$ of $G$, let $\d_e(G)$ denote the graph obtained by smoothing all of the edges of $G$ except for $e$.  Then $\d_e(G)$ is a single-edge graph and therefore corresponds to a boundary divisor, which by \eqref{eq:SGSH} contains $S_G$.

It is convenient to encode the intersections of boundary divisors in the {\it boundary complex} of $\M_{g,n}$.  This is a simplicial complex with a vertex for each boundary divisor and a simplex spanned by a set of vertices whenever their corresponding boundary divisors have nonempty intersection.  To illustrate the idea (and because we will need this example below), we compute the boundary complex of $\M_{2,2}$.

\begin{example}
\label{ex:M22}
There are four boundary divisors in $\M_{2,2}$, and the boundary complex depicting their intersections is as follows:
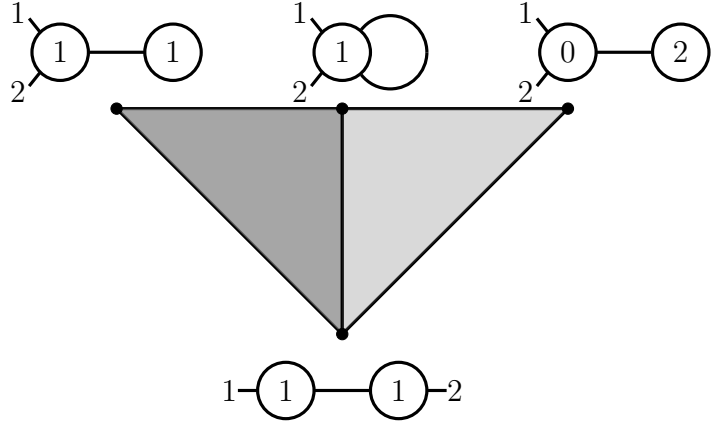
\[\begin{tikzpicture}[scale=0.75]
\draw[very thick] (-4,0) -- (0,0) -- (0,-4) -- (-4,0);
\filldraw[fill= gray, opacity=0.7] (-4,0) -- (0,0) -- (0,-4) -- (-4,0);

\draw[very thick] (4,0) -- (0,0) -- (0,-4) -- (4,0);
\filldraw[fill= gray, opacity=0.3] (4,0) -- (0,0) -- (0,-4) -- (4,0);

\draw[very thick] (5.65-6,0.35+1) -- (5.45-6,0.6+1);
\node at (5.25-6,0.8+0.9){$1$};
\draw[very thick] (5.65-6,-0.35+1) -- (5.45-6,-0.6+1);
\node at (5.25-6,-0.8+1.1){$2$};

\draw[very thick] (0,1) circle (0.50);
\node at (0,1) {$1$};

\draw[very thick] (1.5,1) arc (0:142:0.65);
\draw[very thick] (1.5,1) arc (0:-142:0.65);

\draw[very thick] (4,1) circle (0.50);
\node at (4,1){$0$};

\draw[very thick] (4.5,1) -- (5.5,1);

\draw[very thick] (6,1) circle (0.50);
\node at (6,1){$2$};

\draw[very thick] (10.65-7,0.35+1) -- (10.45-7,0.6+1);
\node at (10.25-7,0.8+0.9){$1$};
\draw[very thick] (10.65-7,-0.35+1) -- (10.45-7,-0.6+1);
\node at (10.25-7,-0.8+1.1){$2$};

\draw[very thick] (4-9,1) circle (0.50);
\node at (4-9,1){$1$};

\draw[very thick] (4.5-9,1) -- (5.5-9,1);

\draw[very thick] (6-9,1) circle (0.50);
\node at (6-9,1){$1$};

\draw[very thick] (10.65-7-9,0.35+1) -- (10.45-7-9,0.6+1);
\node at (10.25-7-9,0.8+0.9){$1$};
\draw[very thick] (10.65-7-9,-0.35+1) -- (10.45-7-9,-0.6+1);
\node at (10.25-7-9,-0.8+1.1){$2$};

\draw[very thick] (-1,-5) circle (0.50);
\node at (-1,-5){$1$};

\draw[very thick] (-0.5,-5) -- (0.5,-5);

\draw[very thick] (1,-5) circle (0.50);
\node at (1,-5){$1$};

\draw[very thick] (-1.5,-5) -- (-1.85,-5);
\node at (-2,-5){$1$};
\draw[very thick] (1.5,-5) -- (1.85,-5);
\node at (2,-5){$2$};

\filldraw (0,0) circle (0.1);
\filldraw (-4,0) circle (0.1);
\filldraw (4,0) circle (0.1);
\filldraw (0,-4) circle (0.1);
\end{tikzpicture}\]
For example, the top-left edge indicates a nonempty intersection of the boundary divisors with the two top-left dual graphs, which is reflected in the fact that these two graphs have
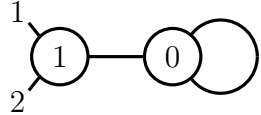
\[\begin{tikzpicture}[scale=0.75]
\draw[very thick] (11,0) circle (0.50);
\node at (11,0){$1$};

\draw[very thick] (11.5,0) -- (12.5,0);

\draw[very thick] (13,0) circle (0.50);
\node at (13,0){$0$};

\draw[very thick] (14.5,0) arc (0:142:0.65);
\draw[very thick] (14.5,0) arc (0:-142:0.65);

\draw[very thick] (10.65,0.35) -- (10.45,0.6);
\node at (10.25,0.8){$1$};
\draw[very thick] (10.65,-0.35) -- (10.45,-0.6);
\node at (10.25,-0.8){$2$};

\end{tikzpicture}\]
as a common degeneration.  Similarly, the two $2$-simplices of the boundary complex indicate nonempty triple intersections of boundary divisors, corresponding to the existence of common degenerations of triples of one-edged dual graphs.
\end{example}

An abstract simplicial complex $\mathcal{K}$ is called a {\it flag complex} if, for any set of vertices $v_1, \ldots, v_k$ of $\mathcal{K}$ in which each pair $\{v_i, v_j\}$ spans a simplex of $\mathcal{K}$, the entire set $\{v_1, \ldots, v_k\}$ spans a simplex of $\mathcal{K}$.  For example, the simplicial complex of Example~\ref{ex:M22} is a flag complex.

We are interested, more generally, in whether the boundary complex of $\M_{g,n}$ is a flag complex for any given $g$ and $n$.  Stated more geometrically, the question is whether, given a collection of boundary divisors $D_1, \ldots, D_k$ in $\M_{g,n}$ such that $D_i \cap D_j \neq \emptyset$ for all $i$ and $j$, it necessarily follows that $D_1 \cap \cdots \cap D_k \neq \emptyset$.  Giansiracusa proved \cite{Giansiracusa} that this property holds when $g=0$, so our task is to study the corresponding statement in higher genus.  Before doing so, we need some preliminary results on boundary divisors.

\subsection{Results on intersections of boundary divisors}

If $D_1, \ldots, D_k$ are distinct boundary divisors in $\M_{g,n}$, then the intersection $D_1 \cap \cdots \cap D_k$ is a union of codimension-$k$ boundary strata, assuming it is nonempty.  To see this, note that if $\xi \in D_i$ is a moduli point corresponding to a curve $(C; x_1, \ldots, x_n)$, then any other moduli point corresponding to a curve with the same dual graph as $(C; x_1, \ldots, x_n)$ also lies in $D_i$.  Thus, if $S_{\xi}$ denotes the set of all curves with the same dual graph as $(C; x_1, \ldots, x_n)$, we have
\[\xi \in D_i \; \Rightarrow \; S_{\xi} \subseteq D_i,\]
and hence $\overline{S}_{\xi} \subseteq D_i$ whenever $\xi \in D_i$, since $D_i$ is closed.  This means that
\begin{equation}
\label{eq:unionS}
\bigcup_{\xi \in D_1 \cap \cdots \cap D_k} \overline{S}_{\xi} \subseteq D_1 \cap \cdots \cap D_k.
\end{equation}
The reverse inclusion is also clearly true, so \eqref{eq:unionS} expresses $D_1 \cap \cdots \cap D_k$ as a union of boundary strata.  Note that, while this appears to be an infinite union, it is in fact finite, since there are only finitely many boundary strata in any given $\M_{g,n}$.

It can happen that this union has multiple components; for example, in $\M_{2,3}$, the intersection of the boundary divisors with dual graphs
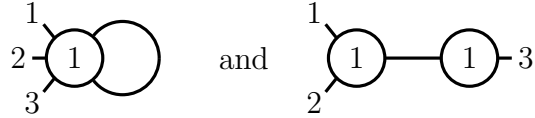
\[\begin{tikzpicture}[scale=0.75]
\draw[very thick] (5.65,0.35) -- (5.45,0.6);
\node at (5.25,0.8){$1$};
\draw[very thick] (5.5,0) -- (5.25,0);
\node at (5,0){$2$};
\draw[very thick] (5.65,-0.35) -- (5.45,-0.6);
\node at (5.25,-0.8){$3$};

\draw[very thick] (6,0) circle (0.50);
\node at (6,0) {$1$};

\draw[very thick] (7.5,0) arc (0:142:0.65);
\draw[very thick] (7.5,0) arc (0:-142:0.65);

\node at (9,0){and};

\draw[very thick] (11,0) circle (0.50);
\node at (11,0){$1$};

\draw[very thick] (11.5,0) -- (12.5,0);

\draw[very thick] (13,0) circle (0.50);
\node at (13,0){$1$};

\draw[very thick] (10.65,0.35) -- (10.45,0.6);
\node at (10.25,0.8){$1$};
\draw[very thick] (10.65,-0.35) -- (10.45,-0.6);
\node at (10.25,-0.8){$2$};

\draw[very thick] (13.5,0) -- (13.75,0);
\node at (14,0){$3$};

\end{tikzpicture}\]
is the union of the two codimension-$2$ boundary strata with dual graphs

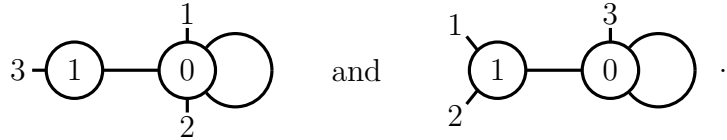
\[\begin{tikzpicture}[scale=0.75]
\draw[very thick] (3.5,0) circle (0.50);
\node at (3.5,0){$1$};

\draw[very thick] (4,0) -- (5,0);

\draw[very thick] (5.5,0) circle (0.50);
\node at (5.5,0){$0$};

\draw[very thick] (7,0) arc (0:142:0.65);
\draw[very thick] (7,0) arc (0:-142:0.65);

\draw[very thick] (3,0) -- (2.75,0);
\node at (2.5,0){$3$};

\draw[very thick] (5.5,0.5) -- (5.5,0.75);
\node at (5.5,1){$1$};
\draw[very thick] (5.5,-0.5) -- (5.5,-0.75);
\node at (5.5,-1){$2$};

\node at (8.5,0){and};

\draw[very thick] (11,0) circle (0.50);
\node at (11,0){$1$};

\draw[very thick] (11.5,0) -- (12.5,0);

\draw[very thick] (13,0) circle (0.50);
\node at (13,0){$0$};

\draw[very thick] (14.5,0) arc (0:142:0.65);
\draw[very thick] (14.5,0) arc (0:-142:0.65);

\draw[very thick] (10.65,0.35) -- (10.45,0.6);
\node at (10.25,0.8){$1$};
\draw[very thick] (10.65,-0.35) -- (10.45,-0.6);
\node at (10.25,-0.8){$2$};

\draw[very thick] (13,0.5) -- (13,0.75);
\node at (13,1){$3$};

\node at (15,0){.};

\end{tikzpicture}\]

\noindent In genus zero, however, a nonempty intersection of boundary divisors is necessarily a single boundary stratum.  This fact is well-known to experts (see, for example, \cite[page 4]{Schaffler}), but we prove it here for completeness, and to allow us to state a slight generalization.  Here, we refer to a boundary stratum as {\it tree-type} if its associated dual graph is a tree, or equivalently, if the dual graph has no {\it nonseparating} edges (edges whose removal leaves the graph connected).

\begin{lemma}\label{lem:uniqueness}
Let $D_1, \ldots, D_k$ be distinct boundary divisors in $\M_{0,n}$.  If $D_1 \cap \cdots \cap D_k \neq \emptyset$, then $D_1 \cap \cdots \cap D_k$ consists of a single boundary stratum.  Furthermore, the same is true in $\M_{1,n}$ if $D_1, \ldots, D_k$ are all tree-type.
\begin{proof}
Let $D_1, \ldots, D_k$ be distinct tree-type boundary divisors in $\M_{0,n}$ or $\M_{1,n}$, with corresponding dual graphs $G_1, \ldots, G_k$, for which $D_1 \cap \cdots \cap D_k \neq \emptyset$.  Then $D_1 \cap \cdots \cap D_k$ is a union of codimension-$k$ boundary strata, and we begin by choosing one such stratum $S$ with corresponding dual graph $G$.  This means that $G_1, \ldots, G_k$ are precisely the dual graphs obtained by smoothing all but one edge of $G$:
\begin{equation}
\label{eq:deG}
\{\d_e(G) \; | \; e \in E(G)\} = \{G_1, \ldots, G_k\},
\end{equation}
where $E(G)$ denotes the set of edges of $G$.  To prove that $S$ is the only codimension-$k$ boundary stratum contained in $D_1 \cap \cdots \cap D_k$, we must prove that $G$ is the only dual graph for which \eqref{eq:deG} holds.

The proof of this claim is by induction on $k$.  It is clearly true when $k=1$, since then both sides of \eqref{eq:deG} consist of the single graph $G$.  Suppose, then, that any graph $H$ with $k-1$ edges is uniquely determined by the graphs $\d_e(H)$ for $e \in E(H)$, and let $G$ be a graph with $k$ edges.

Since $D_1, \ldots, D_k$ are tree-type and the genus is $0$ or $1$, the graph $G$ must have a genus-zero leaf $v$; that is, $v$ is a genus-zero vertex with a unique incident edge.  Let $A \subseteq [n]$ index the legs on $v$, and note that these legs are on the same vertex in any of the graphs $\d_e(G)$.  Furthermore, if $e_1$ denotes the unique edge of $G$ incident to $v$, then $\d_{e_1}(G)$ consists of a genus-zero vertex containing the legs labeled $A$ and another vertex containing the remaining legs.  Using \eqref{eq:deG} and relabeling if necessary so that $\d_{e_1}(G) = G_1$, we thus have the following two observations about the graphs $G_1, \ldots, G_k$:
\begin{enumerate}[label=({\bf O\arabic*})]
\item \label{O1} The legs labeled $A$ lie on the same vertex in any of $G_1, \ldots, G_k$.
\item \label{O2} The graph $G_1$ consists of a genus-zero vertex containing the legs labeled $A$ and another vertex containing the remaining legs.
\end{enumerate}

Now, let $G'$ be any other $k$-edged dual graph satisfying \eqref{eq:deG}; that is,
\begin{equation}
\label{eq:deG'}
\{\d_e(G') \; | \; e \in E(G')\} = \{G_1, \ldots, G_k\}.
\end{equation}
Then the legs labeled $A$ must lie on the same vertex of $G'$, since if there were an edge $e_i'$ separating some of the elements of $A$ from the others, then $\d_{e_i'}(G') = G_i$ would be a graph in which not all of the marked points of $A$ lie on the same vertex, contradicting \ref{O1}.  Furthermore, the vertex $v'$ containing the legs labeled $A$ must be a genus-zero leaf of $G'$ and have no other legs.  To see this, let $e_1' \in E(G')$ be such that $\d_{e_1'}(G') = G_1$.  Then removing $e_1'$ from $G'$ leaves a graph $G' - e_1'$ with two connected components, and there must be some leaf of $G'$ in the same connected component of $G' - e_1'$ as $v'$.  The legs and genus of this leaf are added to those of $v'$ in the graph $\d_{e_1'}(G') = G_1$, so by \ref{O2}, the leaf must be genus zero and be $v'$ itself.

In summary, we have shown that \eqref{eq:deG'} implies
\begin{enumerate}[label=({\bf O\arabic*$'$})]
\item The marked points of $A$ lie on the same vertex $v'$ of $G'$.
\item The vertex $v'$ is a genus-zero leaf of $G'$ containing no other marked points but those of $A$.
\end{enumerate}
The same two facts hold for $G$, by the definition of $A$.  Let $H$ denote the graph obtained from $G$ by deleting the leaf $v$ and replacing its unique incident edge by a leg labeled $\star$.  Let $H'$ denote the analogous graph obtained from $G'$ by deleting $v'$.  Then the graphs $\d_e(H)$ are essentially identical to the graphs $G_2, \ldots, G_k$, except that the legs labeled by $A$ in the graphs $\d_e(G)$ are replaced by the single leg labeled $\star$ in $\d_e(H)$.  The same is true of the graphs $\d_{e'}(H')$.  Thus,
\[\{\d_e(H) \; | \; e \in E(H) \} = \{\d_e(H') \; | \; e \in E(H')\},\]
and since $H$ and $H'$ are both graphs with $k-1$ edges, the inductive hypothesis implies that $H = H'$.  Since $G$ and $G'$ are obtained from $H$ and $H'$, respectively, by attaching a single genus-zero vertex containing the legs labeled $A$ via an edge at leg $\star$, we conclude that $G= G'$.  This completes the proof.
\end{proof}
\end{lemma}

Lemma~\ref{lem:uniqueness} implies that each collection of genus-zero (or genus-one tree-type) boundary divisors $D_1, \ldots, D_k$ with nonempty intersection uniquely determines a boundary stratum $S = D_1 \cap \cdots \cap D_k$.  The next lemma provides something of a converse: each boundary stratum $S$ of $\M_{0,n}$ (or each tree-type boundary stratum of $\M_{1,n}$) uniquely determines a collection of boundary divisors $D_1, \ldots, D_k$ whose intersection is $S$.  Once again, this fact is know to the experts, but we include a proof since one does not seem to be readily available in the literature.

\begin{lemma}\label{lem:uniqueness2}
For any codimension-$k$ boundary stratum $S$ in $\M_{0,n}$, there exists a unique collection of boundary divisors $D_1, \ldots, D_k$ such that $S = D_1 \cap \cdots \cap D_k$.  Furthermore, the same is true in $\M_{1,n}$ if we assume that $S$ is tree-type.
\end{lemma}
\begin{proof}
Let $S$ be a codimension-$k$ tree-type boundary stratum in $\M_{0,n}$ or $\M_{1,n}$, and let $G$ be the corresponding dual graph.  If $E(G) = \{e_1, \ldots, e_k\}$, then the dual graphs $G_i := \d_{e_i}(G)$ each determine a boundary divisor $D_i$ containing $S$.  Thus, we have
\[S \subseteq D_1 \cap \cdots \cap D_k.\]
Assuming that $D_1, \ldots, D_k$ are all distinct, their intersection has codimension $k$, and by Lemma~\ref{lem:uniqueness}, this intersection consists of a single codimension-$k$ boundary stratum, which must therefore be $S$.  Thus, we are done if we can prove that $G_1, \ldots, G_k$ (and hence $D_1, \ldots, D_k$) are distinct.

The proof of this claim is another induction on $k$.  The base case, when $k=1$, is immediate.  Suppose, then, that for any $n$ and any tree-type dual graph for $\M_{0,n}$ or $\M_{1,n}$ with $k-1$ edges, the $k-1$ graphs obtained by smoothing all but one edge are distinct.  Let $G$ be such a dual graph with $k$ edges.

Choose a genus-zero leaf $v$ of $G$, and let $A \subseteq [n]$ be the marked points on $v$.  Let $e_1$ be the unique edge incident to $v$, and let $e_2, \ldots, e_k$ be the remaining edges of $G$.  As in the proof of Lemma~\ref{lem:uniqueness}, let $H$ be the graph obtained from $G$ by deleting $v$ and replacing $e_1$ by a leg that we label $\star$.  Then $H$ has $k-1$ edges, identified with the edges $e_2, \ldots, e_k$ of $G$, so by induction, the graphs
\[\d_{e_2}(H), \ldots, \d_{e_k}(H)\]
are all distinct.  By replacing the leg labeled $\star$ with the legs labeled $A$, we obtain the graphs
\[\d_{e_2}(G), \ldots, \d_{e_k}(G),\]
and hence these are also distinct.  Furthermore, they are distinct from $\d_{e_1}(G)$, because $\d_{e_1}(G)$ has a genus-zero vertex containing only the marked points of $A$.  If this were true of $\d_{e_i}(G)$ for some $i \neq 1$, then $\d_{e_i}(H)$ would have a genus-zero vertex containing only the marked point $\star$, contradicting stability.  Thus, $\d_{e_1}(G), \ldots, \d_{e_k}(G)$ are all distinct, as claimed.
\end{proof}

We remark that the tree-type assumption in Lemma~\ref{lem:uniqueness2} is indeed necessary.  For example, in $\M_{1,2}$, the two graphs obtained by smoothing one edge of
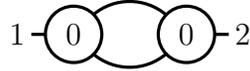
\begin{equation}
\nonumber
\begin{tikzpicture}[scale=0.75]

\draw[very thick]  (0,0) circle (0.50);
\node at (0,0) {$0$};

\draw[very thick] (2,0) circle (0.50);
\node at (2,0) {$0$};

\draw[very thick] (0.35,0.35) arc (130:50:1);
\draw[very thick] (0.35,-0.35) arc (230:310:1);

\draw[very thick] (-0.5,0) -- (-0.75,0);
\node at (-1,0){$1$};

\draw[very thick] (2.5,0) -- (2.75,0);
\node at (3,0){$2$};
\end{tikzpicture}
\end{equation}
then the two graphs obtained by smoothing the two edges of $G$ are identical.

\section{The boundary complex in higher genus}

Equipped with this background, we are ready to begin the proof of Theorem~\ref{thm:main}.  The genus-zero statement is Giansiracusa's work~\cite{Giansiracusa}, so we begin with the case $g=1$, in which case the existence of the moduli space requires that $n \geq 1$.

\begin{lemma}\label{lem:g=1}
The boundary complex of $\M_{1,n}$ is a flag complex for any $n \geq 1$.
\end{lemma}

The idea of the proof of Lemma~\ref{lem:g=1} is to relate boundary strata in genus one to boundary strata in genus zero.  Namely, let $\mathcal{S}_{g,n}$ be the set of boundary strata in $\M_{g,n}$, and let
\[\S_{1,n} :=\{\text{tree-type boundary strata}\} \subseteq \mathcal{S}_{1,n}.\]
Equivalently, the dual graphs of strata in $\S_{1,n}$ are those that have a genus-one vertex.  There is a map that takes $\S_{1,n}$ to $\mathcal{S}_{0,n+2}$ by replacing the genus-one vertex with a genus-zero vertex to which we add marked points $n+1$ and $n+2$, and the image of this map is the set
\[\S_{0,n+2} \subseteq \mathcal{S}_{0,n+2}\]
consisting of all boundary strata whose corresponding dual graph has marked points $n+1$ and $n+2$ on the same vertex.  The process is clearly reversible, so the result is a bijection
\[\sigma: \S_{1,n} \rightarrow \S_{0,n+2}.\]
This bijection is inclusion-preserving in both directions, since a degeneration of dual graphs in $\S_{1,n}$ induces a corresponding degeneration of dual graphs in $\S_{0,n+2}$ and vise versa.

Furthermore, Lemmas~\ref{lem:uniqueness} and \ref{lem:uniqueness2} imply that both the domain and codomain of $\sigma$ are closed under intersection of boundary strata.  To see this, suppose $D_1 \ldots, D_k \in \S_{1,n}$ and $D_1 \cap \cdots \cap D_k \neq \emptyset$.  Then by Lemma~\ref{lem:uniqueness}, the intersection $D_1 \cap \cdots \cap D_k$ is a single boundary stratum $S$.  If the dual graph $G$ corresponding to $S$ had a nonseparating edge, then smoothing all but this edge would yield a boundary divisor containing $S$ that is not among $D_1, \ldots, D_k$, contradicting the uniqueness of the expression $S = D_1 \cap \cdots \cap D_k$ as an intersection of boundary divisors given by Lemma~\ref{lem:uniqueness2}.  The argument is similar for the codomain: if $D_1', \ldots, D_k' \in \S_{0,n+2}$ and $D_1' \cap \cdots \cap D_k' \neq \emptyset$, then $D_1' \cap \cdots \cap D_k'$ is again a single boundary stratum $S'$ with corresponding dual graph $G'$.  If marked points $n+1$ and $n+2$ were on different vertices of $G'$, then there would be an edge separating them, and smoothing all but this edge would yield a boundary divisor containing $S'$ that is not among $D_1' , \ldots, D_k'$, a contradiction.

From here, it is not hard to see that, for any collection of $k$ distinct boundary divisors $D_1, \ldots, D_k \in \S_{1,n}$, we have
\begin{equation}
\label{eq:posetiso}
\sigma(D_1 \cap \cdots \cap D_k) = \sigma(D_1) \cap \cdots \cap \sigma(D_k).
\end{equation}
To see this, suppose first that $D_1 \cap \cdots \cap D_k \neq \emptyset$.  Then by Lemma~\ref{lem:uniqueness} and the previous paragraph, the intersection $D_1 \cap \cdots \cap D_k$ is a single codimension-$k$ boundary stratum $S \in \S_{1,n}$.  Since $S \subseteq D_i$ for all $i$, and $\sigma$ is inclusion-preserving, we have $\sigma(S) \subseteq \sigma(D_i)$ for all $i$ and hence
\[\sigma(S) \subseteq \sigma(D_1) \cap \cdots \cap \sigma(D_k).\]
But both $\sigma(S)$ and $\sigma(D_1) \cap \cdots \cap \sigma(D_k)$ are codimension-$k$ boundary strata, so the above containment must be an equality.

If, on the other hand, $D_1 \cap \cdots \cap D_k = \emptyset$, then the claim that \eqref{eq:posetiso} holds is equivalent to proving that 
\[\sigma(D_1) \cap \cdots \cap \sigma(D_k) = \emptyset.\]
Suppose, toward a contradiction, that $\sigma(D_1) \cap \cdots \cap \sigma(D_k) \neq \emptyset$, in which case Lemma~\ref{lem:uniqueness} implies that $\sigma(D_1) \cap  \cdots \cap \sigma(D_k) = S'$ for some codimension-$k$ boundary stratum $S'$.  Since the codomain of $\sigma$ is closed under intersection, this means that
\[\sigma(D_1) \cap \cdots \cap \sigma(D_k) = \sigma(S)\]
for some stratum $S$, meaning that $\sigma(S) \subseteq \sigma(D_i)$ for all $i$.  The fact that $\sigma^{-1}$ is inclusion-preserving implies $S \subseteq D_i$ for all $i$, which contradicts the assumption that $D_1 \cap \cdots \cap D_k = \emptyset$.

With these observations in place, the proof of Lemma~\ref{lem:g=1} is not far behind.

\begin{proof}[Proof of Lemma~\ref{lem:g=1}]
Let $D_1, \ldots, D_k$ be boundary divisors in $\M_{1,n}$ with $D_i \cap D_j \neq \emptyset$ for all $i$ and $j$; our goal is to prove that $D_1 \cap \cdots \cap D_k \neq \emptyset$.

Assume, for now, that $D_1, \ldots, D_k$ all come from $\S_{1,n}$.  In this case, we can apply $\sigma$ to obtain a collection of boundary divisors
\[\sigma(D_1), \ldots, \sigma(D_k) \in \S_{0,n+2}\]
for which, by \eqref{eq:posetiso}, we have
\[\sigma(D_i) \cap \sigma(D_j) = \sigma(D_i \cap D_j) \neq \emptyset\]
for all $i$ and $j$.  The fact that the boundary complex of $\M_{0,n+2}$ is a flag complex then implies
\[\sigma(D_1) \cap \cdots \cap \sigma(D_k) \neq \emptyset,\]
so by \eqref{eq:posetiso} again, we have
\[\sigma(D_1 \cap \cdots \cap D_k) \neq \emptyset\]
and hence $D_1 \cap \cdots \cap D_k \neq \emptyset$.

On the other hand, suppose that one of $D_1, \ldots, D_k$ (without loss of generality, say $D_k$) is the divisor $D_0$ whose dual graph $G_0$ consists of a single genus-zero vertex with a self-edge.  Note that this divisor intersects nontrivially with every boundary stratum of $\M_{1,n}$.  Indeed, if $S$ is a boundary stratum with associated dual graph $G$, then either $G$ has a nonseparating edge or $G$ has a genus-one vertex.  In the first case, $G$ is a degeneration of $G_{0}$ (as one sees by smoothing all but the nonseparating edge of $G$), so $S \subseteq D_{0}$.  In the second case, if we replace the genus-one vertex of $G$ by a genus-zero vertex with a self-edge, then the resulting graph is a degeneration of both $G_{0}$ and $G$, so $D_{0} \cap S \neq \emptyset$.

In particular, then, since our previous argument shows that $D_1 \cap \cdots \cap D_{k-1} \neq \emptyset$, it follows from Lemma~\ref{lem:uniqueness} that $D_1 \cap \cdots \cap D_{k-1}$ is a single stratum $S$.  We then have
\[D_1 \cap \cdots \cap D_k = S \cap D_k = S \cap D_0 \neq \emptyset\]
by the above, so the proof is complete.
\end{proof}

We now turn to the cases of Theorem~\ref{thm:main} where $n=0$ or $n=1$, in which the claim is that the boundary complex of $\M_{g,n}$ is always a flag complex.

\begin{lemma}
\label{lem:n=0or1}
The boundary complex of $\M_{g,0}$ and of $\M_{g,1}$ is a flag complex for any $g$ for which the moduli space exists.
\end{lemma}
\begin{proof}
Consider the moduli space $\M_{g,0}$, where $g \geq 2$ so that the moduli space exists.  Then the dual graph
\begin{equation}
\label{eq:n=0graph}
\begin{tikzpicture}[scale=0.75]

\draw[very thick]  (0,0) circle (0.50);
\node at (0,0) {$1$};

\draw[very thick] (0.50,0) -- (1,0);

\draw[very thick] (1.5,0) circle (0.50);
\node at (1.5,0) {$1$};

\draw[very thick] (2,0) -- (2.5,0);

\node at (3,0){$\cdots$};

\draw[very thick] (3.5,0) -- (4,0);

\draw[very thick] (4.5,0) circle (0.50);
\node at (4.5,0) {$1$};

\draw[very thick] (5,0) -- (5.5,0);

\draw[very thick] (6,0) circle (0.50);
\node at (6,0) {$0$};

\draw[very thick] (7.5,0) arc (0:142:0.65);
\draw[very thick] (7.5,0) arc (0:-142:0.65);

\end{tikzpicture}
\end{equation}
is a degeneration of the dual graph of any boundary divisor of $\M_{g,0}$.  Indeed, it is straightforward to a give a full list of the dual graphs of the boundary divisors of $\M_{g,0}$: there is the single-vertex graph $G_0$ with a self-edge, and for any partition $a+b=g$ in which $a,b>0$, there is a graph $G_{a,b}$ consisting of a vertex of genus $a$ joined by an edge to a vertex of genus $b$.  One sees that $G$ is a degeneration of $G_0$ by smoothing all but the self-edge of $G$, while one sees that $G$ is a degeneration of $G_{a,b}$ by smoothing the leftmost $a$ edges and the rightmost $b$ edges (including the self-edge) of $G$.

It follows that, for any collection $D_1, \ldots, D_k$ of boundary divisors of $\M_{g,0}$, we have
\[D_1 \cap \cdots \cap D_k \neq \emptyset,\]
because the boundary stratum corresponding to $G$ is contained in $D_i$ for each $i$.  Thus, the boundary complex of $\M_{g,0}$ is a flag complex.

A similar argument applies to $\M_{g,1}$ (where now we need only assume that $g \geq 1$).  In this case, the same argument shows that the dual graph
\[\begin{tikzpicture}[scale=0.75]

\draw[very thick]  (0,0) circle (0.50);
\node at (0,0) {$1$};

\draw[very thick] (0.50,0) -- (1,0);

\draw[very thick] (1.5,0) circle (0.50);
\node at (1.5,0) {$1$};

\draw[very thick] (2,0) -- (2.5,0);

\node at (3,0){$\cdots$};

\draw[very thick] (3.5,0) -- (4,0);

\draw[very thick] (4.5,0) circle (0.50);
\node at (4.5,0) {$1$};

\draw[very thick] (5,0) -- (5.5,0);

\draw[very thick] (6,0) circle (0.50);
\node at (6,0) {$0$};
\draw[very thick] (6,0.5) -- (6,0.75);
\node at (6,1.05) {$1$};

\draw[very thick] (7.5,0) arc (0:142:0.65);
\draw[very thick] (7.5,0) arc (0:-142:0.65);

\end{tikzpicture}\]
is a degeneration of the dual graph of any boundary divisor.
\end{proof}

It is worth remarking that the proof of Lemma~\ref{lem:n=0or1} does not extend to $n \geq 2$, since in this case, there are dual graphs $G_{a,b}$ with a vertex of genus zero, and no graph of the shape in \eqref{eq:n=0graph} is not a degeneration of such $G_{a,b}$.  Indeed, we will see below that when $g \geq 3$, the boundary complex of $\M_{g,n}$ is not a flag complex for any $n \geq 2$.  Before that, however, we address the slightly special case of genus two.

\begin{lemma}
\label{lem:g=2}
The boundary complex of $\M_{2,n}$ is a flag complex if and only if $n \in \{0,1,2\}$.
\end{lemma}
\begin{proof}
The cases $n=0$ and $n=1$ are both covered by Lemma~\ref{lem:n=0or1}.  If $n=2$, we computed the boundary complex in Example~\ref{ex:M22}, and we can see at a glance that it is a flag complex.

Suppose, then, that $n \geq 3$.  For each $i \in \{1,\ldots, n\}$, let $D_i$ be the boundary divisor associated to the dual graph

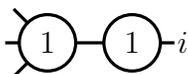
\begin{center}
\begin{tikzpicture}[scale=0.75]
\draw[very thick]  (0,0) circle (0.5);
\node at (0,0) {$ 1$};

\draw[very thick] (0.50,0) -- (1,0);

\draw[very thick] (1.5,0) circle (0.50);
\node at (1.5,0) {$1$};

\draw[very thick] (-0.35,0.35) -- (-0.6,0.6);
\draw[very thick] (-0.5,0) -- (-0.75,0);
\draw[very thick] (-0.35,-0.35) -- (-0.6,-0.6);

\draw[very thick] (2,0) -- (2.25,0);
\node at (2.4,0) {$i$};
\end{tikzpicture}
\end{center}

\noindent with the left-hand component containing all marked points except $i$.  Then $D_i \cap D_j$ has dual graph
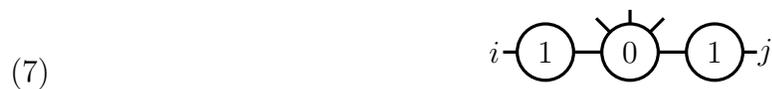
\begin{equation}
\label{eq:Dij}
\begin{tikzpicture}[scale=0.75]
\draw[very thick]  (-1.5,0) circle (0.50);
\node at (-1.5,0) {$ 1$};

\draw[very thick]  (0,0) circle (0.50);
\node at (0,0) {$0$};

\draw[very thick] (0.50,0) -- (1,0);
\draw[very thick] (-0.50,0) -- (-1,0);

\draw[very thick] (1.5,0) circle (0.50);
\node at (1.5,0) {$1$};

\draw[very thick] (-2,0) -- (-2.25,0);
\node at (-2.4,0) {$i$};

\draw[very thick] (0-0.35,0+0.35) -- (0-0.6,0+0.6);
\draw[very thick] (0,0.5) -- (0,0.75);
\draw[very thick] (0+0.35,0+0.35) -- (0+0.6,0+0.6);

\draw[very thick] (2,0) -- (2.25,0);
\node at (2.4,0) {$j$};
\end{tikzpicture}
\end{equation}
so it is in particular nonempty.

We claim, however, that
\[D_1 \cap \cdots \cap D_n = \emptyset.\]
To see this, note that if $D_1 \cap \cdots \cap D_n$ were nonempty, then it would have codimension $n$, and hence there would exist a dual graph $G$ with $n$ edges that is a degeneration of the dual graphs of $D_i$ for each $i$.  This, in particular, means that $G$ is a degeneration of the graph in \eqref{eq:Dij} for any $i \neq j$, which means that marked points $i$ and $j$ cannot lie on the same vertex or on adjacent vertices of $G$.

Given that $G$ has $n$ vertices, it follows from Euler's formula that
\[|V(G)| + |F(G)|=n+1,\]
where $V(G)$ and $F(G)$ denote the sets of vertices and faces of $G$, respectively.  If $|F(G)| \geq 1$, this implies that $|V(G)| \leq n$, in which case it is impossible to distribute the $n$ marked points without some pair of marked points lying on the same vertex.  The only possibility, then, is that $|F(G)| = 0$, in which case $|V(G)| = n+1$.  There is now exactly one way to distribute the $n$ marked points so that no two lie on adjacent vertices: the graph $G$ must have a ``pinwheel" shape, with one central unmarked vertex and $n$ other vertices each connected to it by an edge and each containing one marked point.  But since $g=2$ and $n \geq 3$, at least one of these exterior vertices must have genus zero, and it is therefore unstable.

We conclude that there is no graph that is a degeneration of the dual graphs of $D_i$ for all $i$, and hence the boundary complex of $\M_{2,n}$ is not a flag complex.
\end{proof}

The remaining cases of Theorem~\ref{thm:main}, in which $g \geq 3$, can be handled all together.  In these cases, aside from the small values of $n$ covered by Lemma~\ref{lem:n=0or1}, the boundary complex is never a flag complex.

\begin{lemma}
\label{lem:ggeq3}
Let $g \geq 3$.  Then the boundary complex of $\M_{g,n}$ is a flag complex if and only if $n \in \{0,1\}$.
\end{lemma}
\begin{proof}
In light of Lemma~\ref{lem:n=0or1}, it suffices to prove that the boundary complex of $\M_{g,n}$ is never a flag complex when $g \geq 3$ and $n \geq 2$.  To do so, let $D_1$, $D_2$, and $D_3$ be the boundary divisors specified by the following three dual graphs:

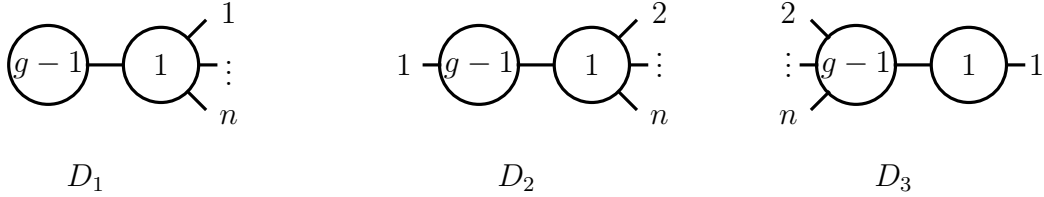
\begin{center}
\begin{minipage}{0.3\linewidth}
\begin{tikzpicture}
\draw[very thick]  (0,0) circle (0.525);
\node at (0,0) {$ g-1$};

\draw[very thick] (0.50,0) -- (1,0);

\draw[very thick] (1.5,0) circle (0.50);
\node at (1.5,0) {$1$};

\draw[very thick] (1.5+0.35,0+0.35) -- (1.5+0.6,0+0.6);
\node at (1.5+0.9,0+0.7){$1$};
\draw[very thick] (2,0) -- (2.25,0);
\node at (2.4,0) {$\vdots$};
\draw[very thick] (1.5+0.35,0-0.35) -- (1.5+0.6, 0-0.6);
\node at (1.5+0.9, 0-0.7) {$n$};

\node at (0.5,-1.5){$D_1$};
\end{tikzpicture}
\end{minipage}
\begin{minipage}{0.3\linewidth}
\begin{tikzpicture}
\draw[very thick]  (0,0) circle (0.525);
\node at (0,0) {$ g-1$};

\draw[very thick] (0.50,0) -- (1,0);

\draw[very thick] (1.5,0) circle (0.50);
\node at (1.5,0) {$1$};

\draw[very thick] (-0.5,0) -- (-0.75,0);
\node at (-1,0) {$1$};

\draw[very thick] (1.5+0.35,0+0.35) -- (1.5+0.6,0+0.6);
\node at (1.5+0.9,0+0.7){$2$};
\draw[very thick] (2,0) -- (2.25,0);
\node at (2.4,0.1) {$\vdots$};
\draw[very thick] (1.5+0.35,0-0.35) -- (1.5+0.6, 0-0.6);
\node at (1.5+0.9, 0-0.7) {$n$};

\node at (0.5,-1.5){$D_2$};
\end{tikzpicture}
\end{minipage}
\begin{minipage}{0.3\linewidth}
\begin{tikzpicture}
\draw[very thick]  (0,0) circle (0.525);
\node at (0,0) {$ g-1$};

\draw[very thick] (0.50,0) -- (1,0);

\draw[very thick] (1.5,0) circle (0.50);
\node at (1.5,0) {$1$};

\draw[very thick] (-0.35,0.35) -- (-0.6,0.6);
\node at (-0.9,0.7) {$2$};
\draw[very thick] (-0.5,0) -- (-0.75,0);
\node at (-0.9,0.1) {$\vdots$};
\draw[very thick] (-0.35,-0.35) -- (-0.6,-0.6);
\node at (-0.9,-0.7) {$n$};

\draw[very thick] (2,0) -- (2.25,0);
\node at (2.4,0) {$1$};

\node at (0.5,-1.5){$D_3$};
\end{tikzpicture}
\end{minipage}
\end{center}

\noindent One can check that the pairwise intersections have dual graphs
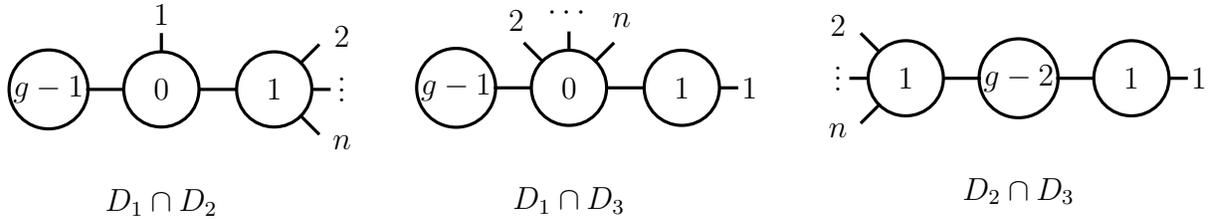
\begin{center}
\begin{minipage}{0.32\linewidth}
\begin{tikzpicture}
\draw[very thick]  (-1.5,0) circle (0.525);
\node at (-1.5,0) {$ g-1$};

\draw[very thick]  (0,0) circle (0.50);
\node at (0,0) {$0$};

\draw[very thick] (0.50,0) -- (1,0);
\draw[very thick] (-0.50,0) -- (-1,0);

\draw[very thick] (1.5,0) circle (0.50);
\node at (1.5,0) {$1$};

\draw[very thick] (0,0.5) -- (0,0.75);
\node at (0,1){$1$};

\draw[very thick] (1.5+0.35,0+0.35) -- (1.5+0.6,0+0.6);
\node at (1.5+0.9,0+0.7){$2$};
\draw[very thick] (2,0) -- (2.25,0);
\node at (2.4,0.1) {$\vdots$};
\draw[very thick] (1.5+0.35,0-0.35) -- (1.5+0.6, 0-0.6);
\node at (1.5+0.9, 0-0.7) {$n$};

\node at (0,-1.5){$D_1 \cap D_2$};
\end{tikzpicture}
\end{minipage}
\begin{minipage}{0.32\linewidth}
\begin{tikzpicture}
\draw[very thick]  (-1.5,0) circle (0.5250);
\node at (-1.5,0) {$ g-1$};

\draw[very thick]  (0,0) circle (0.50);
\node at (0,0) {$0$};

\draw[very thick] (0.50,0) -- (1,0);
\draw[very thick] (-0.50,0) -- (-1,0);

\draw[very thick] (1.5,0) circle (0.50);
\node at (1.5,0) {$1$};

\draw[very thick] (0-0.35,0+0.35) -- (0-0.6,0+0.6);
\node at (0-0.7, 0+0.9){$2$};
\draw[very thick] (0,0.5) -- (0,0.75);
\node at (0,1){$\cdots$};
\draw[very thick] (0+0.35,0+0.35) -- (0+0.6,0+0.6);
\node at (0+0.7, 0+0.9){$n$};

\draw[very thick] (2,0) -- (2.25,0);
\node at (2.4,0) {$1$};

\node at (0,-1.5){$D_1 \cap D_3$};
\end{tikzpicture}
\end{minipage}
\begin{minipage}{0.3\linewidth}
\begin{tikzpicture}
\draw[very thick]  (-1.5,0) circle (0.50);
\node at (-1.5,0) {$1$};

\draw[very thick]  (0,0) circle (0.525);
\node at (0,0) {$g-2$};

\draw[very thick] (1.5,0) circle (0.5);
\node at (1.5,0) {$1$};

\draw[very thick] (0.50,0) -- (1,0);
\draw[very thick] (-0.50,0) -- (-1,0);

\draw[very thick] (-1.5-0.35,0+0.35) -- (-1.5-0.6,0+0.6);
\node at (-1.5-0.9,0+0.7){$2$};
\draw[very thick] (-2,0) -- (-2.25,0);
\node at (-2.4,0.1) {$\vdots$};
\draw[very thick] (-1.5-0.35,0-0.35) -- (-1.5-0.6, 0-0.6);
\node at (-1.5-0.9, 0-0.7) {$n$};

\draw[very thick] (2,0) -- (2.25,0);
\node at (2.4,0) {$1$};

\node at (0,-1.5){$D_2 \cap D_3$};
\end{tikzpicture}
\end{minipage}
\end{center}
\noindent which confirms that all three pairwise intersections are nonempty.  However, we claim that
\[D_1 \cap D_2 \cap D_3 = \emptyset,\]
which implies that the boundary complex of $\M_{g,n}$ is not a flag complex.

To check this, note that since the middle vertex of the dual graph for $D_1 \cap D_2$ has genus zero and three special points, it cannot degenerate.  Therefore, in any degeneration of that graph, removing the vertex containing marked point $1$ yields a disjoint union of two graphs, one of genus $g-1$ containing no marked points and one of genus $1$ containing marked points $2, \ldots, n$.  This is never true of a degeneration of the dual graph for $D_1 \cap D_3$: in that case, removing the vertex with marked point $1$ either produces a connected graph, or it produces a disjoint union in which one subgraph has genus $g-1$ and contains marked points $2, \ldots, n$.

Thus, no dual graph can be simultaneously a degeneration of the dual graph of $D_1 \cap D_2$ and of $D_1 \cap D_3$, so the triple intersection is empty.
\end{proof}

Combining Lemmas~\ref{lem:g=1}, ~\ref{lem:n=0or1},~\ref{lem:g=2}, and~\ref{lem:ggeq3}, the proof of Theorem~\ref{thm:main} is complete.

\bibliographystyle{abbrv}
\bibliography{biblio}

\end{document}